\documentclass[12pt]{article}

\usepackage{dsfont}
\usepackage{amsfonts,amsmath,amsthm}
\usepackage{algorithm, algorithmic}
\usepackage{color}
\usepackage{graphicx}
\usepackage{stmaryrd}        

\usepackage[paper=a4paper,dvips,top=2cm,left=2cm,right=2cm,
    foot=1cm,bottom=4cm]{geometry}

\newcommand{\blue}[1]{\begin{color}{blue}#1\end{color}}

\begin{document}

\title{Quaternion Matrix Optimization and The Underlying Calculus}
\author{ Liqun Qi\footnote{%
    Department of Applied Mathematics, The Hong Kong Polytechnic University,
    Hung Hom, Kowloon, Hong Kong ({\tt maqilq@polyu.edu.hk}).}
    \and \
    Ziyan Luo\footnote{Department of Mathematics,
  Beijing Jiaotong University, Beijing 100044, China. (zyluo@bjtu.edu.cn). This author's work was supported by NSFC (Grant No.  11771038) and Beijing Natural Science Foundation (Grant No.  Z190002).}
  \and \
    Qing-Wen Wang\footnote{Department of Mathematics,
  Shanghai University, Shanghai 200444, China. (wqw@t.shu.edu.cn). This author's work was supported by NSFC (Grant No.  11971294).}
  \and and \
    Xinzhen Zhang\thanks{Corresponding author.   School of Mathematics, Tianjin University, Tianjin 300354 China; ({\tt xzzhang@tju.edu.cn}). This author's work was supported by NSFC (Grant No.  11871369). }
}
\date{\today}
\maketitle

\begin{abstract}

Optimization models involving quaternion matrices are widely used in color image process and other engineering areas.   These models optimize real functions of quaternion matrix variables.   In particular, $\ell_0$-norms and rank functions of quaternion matrices are discrete.   Yet calculus with derivatives, subdifferentials and generalized subdifferentials of such real functions is needed to handle such models.   In this paper, we introduce first and second order derivatives and establish their calculation rules for such real functions.   Our approach is consistent with the subgradient concept for norms of quaternion matrix variables, recently introduced in the literature.    We develop the concepts of  generalized subdifferentials of proper functions of quaternion matrices, and use them to analyze the optimality conditions of a sparse low rank color image denoising model.
We introduce R-product for two quaternion matrix vectors, as a key tool for our calculus.   We show that the real representation set of low-rank quaternion matrices is closed and semi-algebraic.    
We also establish first order and second order optimality conditions for constrained optimization problems of real functions in quaternion matrix variables.

  \medskip


\end{abstract}

\renewcommand{\Re}{\mathds{R}}
\newcommand{\rank}{\mathrm{rank}}
\renewcommand{\span}{\mathrm{span}}
\newcommand{\X}{\mathcal{X}}
\newcommand{\A}{\mathcal{A}}
\newcommand{\B}{\mathcal{B}}
\newcommand{\C}{\mathcal{C}}
\newcommand{\OO}{\mathcal{O}}
\newcommand{\e}{\mathbf{e}}
\newcommand{\0}{\mathbf{0}}
\newcommand{\dd}{\mathbf{d}}
\newcommand{\ii}{\mathbf{i}}
\newcommand{\jj}{\mathbf{j}}
\newcommand{\kk}{\mathbf{k}}
\newcommand{\va}{\mathbf{a}}
\newcommand{\vb}{\mathbf{b}}
\newcommand{\vc}{\mathbf{c}}
\newcommand{\vg}{\mathbf{g}}
\newcommand{\vr}{\mathbf{r}}
\newcommand{\vt}{\rm{vec}}
\newcommand{\vx}{\mathbf{x}}
\newcommand{\vy}{\mathbf{y}}
\newcommand{\y}{\mathbf{y}}
\newcommand{\vz}{\mathbf{z}}
\newcommand{\T}{\top}

\newtheorem{Thm}{Theorem}[section]
\newtheorem{Def}[Thm]{Definition}
\newtheorem{Ass}[Thm]{Assumption}
\newtheorem{Lem}[Thm]{Lemma}
\newtheorem{Prop}[Thm]{Proposition}
\newtheorem{Cor}[Thm]{Corollary}

\section{introduction}

Quaternion matrix methods have been widely used in color image processing, including color image denoising and inpainting, color face recognization, etc., \cite{CXZ19, JNS19, LS03, Sa96, SL06, XZ18, XYXZN15, ZKW16}.   Quaternions have also been widely used in the other engineering areas \cite{XJTM15}.   Optimization techniques are frequently used in this process.

Various real functions of quaternion matrix variables, such as Frobenius norms, nuclear norms, spectral norms, $\ell_1$-norms, traces, $\ell_0$-norms and rank functions of quaternion matrices, arise from engineering applications.

To optimize such real functions, derivatives, subdifferentials and generalized subdifferentials are needed to handle them.

For the color image denoising problem, optimization models may involve $\ell_0$-norms and rank functions of
quaternion matrix variables.   They are discrete.   In sparse optimization \cite{BE13, LXS19}, generalized subdifferentials of $\ell_0$-norms and rank functions of real matrix variables are used.  This uses the knowledge of variational analysis \cite{Cl83, HL13, Le13, Mo06, RW09}.    In this paper, we aim to develop adequate analysis tools for developing derivatives, subdifferentials and generalized subdifferentials of real functions of quaternion matrix variables, such that we can establish optimality conditions of such optimization models and pave the way for future works on convergence analysis of algorithms for such optimization models.

In the literature, there are some works on derivatives of real functions of quaternion variables \cite{MJT11, XJTM15}, and subgradients of norms of quaternion matrix variables \cite{JNS19}.   There are no discussion on generalized subdifferentials of $\ell_0$-norms and rank functions of quaternion matrix variables.


In \cite{CQZX20}, Chen, Qi, Zhang and Xu formulated the color image inpainting problem as an equality constrained optimization problem of real functions in quaternion matrix variables, and proposed a lower rank quaternion decomposition (LRQD) algorithm to solve it.  To conduct convergence analysis for their algorithm, they introduced a concise form
$$\nabla f(X) = {\partial f \over \partial X_0} + {\partial f \over \partial X_1}\ii + {\partial f \over \partial X_2}\jj + {\partial f \over \partial X_3}\kk$$
for the gradient of a real function $f$ in a quaternion matrix variable $X=X_0 + X_1\ii + X_2\jj + X_3 \kk$.  This form is different from the generalized HR calculus studied in \cite{XJTM15}.   The first order optimality necessary condition of their quaternion least squares problem has a simple expression with this form.   With this tool, convergence and convergence rate of their algorithm are established.

This motivates us to study constrained optimization of real functions in quaternion matrix variables systematically.   We consider general inequality and equality constrained optimization problems of real functions in quaternion matrix variables.   We use the LRQD problem as a prototype problem.  We introduce R-product and R-linear independence of quaternion matrix vectors, and establish the first order necessary optimality condition for the general inequality and equality constrained optimization problem of real functions in quaternion matrix variables.   We present a method to calculate the second order partial derivatives of real functions in quaternion matrix variables, and establish the second order necessary optimality condition and second order sufficient optimality condition for the general constrained optimization problem of real functions in quaternion matrix variables.

Norms of quaternion matrix variables may not be continuously differentiable, but they are always convex. In 2019, Jia, Ng and Song \cite{JNS19} introduced subgradients for norms of quaternion matrix variables.

In this paper, we show that our approach is consistent with the subgradient concept for norms of quaternion matrix variables, introduced in \cite{JNS19}.   We also establish the relations between subgradients and derivatives of real functions of quaternion matrix variables.

The $\ell_0$-norms and the ranks of quaternion matrices are not continuous at all, but they are lower semi-continuous.   They are very useful in applications \cite{CXZ19, XYXZN15}.   In this paper, we develop the concepts of  generalized subdifferentials of proper functions of quaternion matrices, and use them to analyze the optimality conditions of a sparse low rank color image denoising model.  We show that the real representation set of low-rank quaternion matrices is closed and semi-algebraic.

The generalized subdifferential calculus is totally new in the literature and useful in color image applications.

While some important real functions, such as the squares of the Frobenius norms, of quaternion matrix variables are separable in the sense that they can be calculated with respect to each real matrix variable, then summed together, the rank function of quaternion matrix variables is not separable.   We treat this in a novel way by considering the real representation of a quaternion matrix, and show that the real representation set of low-rank quaternion matrices is semi-algebraic.   This will be useful for convergence analysis of some first order algorithms.

In the next section, we present some necessary preliminary knowledge on quaternions and quaternion matrices.   The general inequality and equality constrained optimization problem of real functions in quaternion matrix variables and its LRQD prototype problem are presented in Section 3.
In Section 4, we introduce R-product, R-linear independence and present first order necessary optimality condition for the general constrained optimization problem.   We also establish the product rule and the chain rule of first order derivatives.  We study second order derivatives and present second order optimality conditions in Section 5.    
In Section 6, we study convex functions of quaternion matrix variables, their subdifferentials, and the relations with our derivatives of real functions of
quaternion matrix variables.    We introduce the generalized subdifferentials of proper functions of quaternion matrices, and use them to analyze the optimality conditions of a sparse low rank color image denoising model in Section 7.   Some final remarks are made in Section 8.

\section{Preliminary}

\subsection{Quaternions}

In general, we use the notation in \cite{Zh97, WLZZ18}.  We denote the real number field, the complex number field and the quaternion algebra by ${\mathbb R}$, $\mathbb C$ and $\mathbb Q$, respectively.  Scalars, vectors, matrices and tensors are denoted by small letters, bold small letters, capital letters and calligraphic letters, respectively.   Vectors with matrix components are denoted by bold capital letters.   For example, we have ${\bf X} = (W, Y, Z)$.   They are called matrix component vectors.
We use $\0, O$, $\OO$ and $\bf O$ to denote zero vector, zero matrix, zero tensor and zero matrix component vector with adequate dimensions.  The three imaginary units of quaternions are denoted by
$\ii, \jj$ and $\kk$.    We have
$$\ii^2 = \jj^2 = \kk^2 =\ii\jj\kk = -1,$$
$$\ii\jj = -\jj\ii = \kk, \ \jj\kk = - \kk\jj = \ii, \kk\ii = -\ii\kk = \jj.$$
These rules, along with the distribution law, determine the product of two quaternions.   The multiplication of quaternions is noncommutative.

Let $x = x_0 + x_1\ii + x_2\jj + x_3\kk \in {\mathbb Q}$, where $x_0, x_1, x_2, x_3 \in {\mathbb R}$.
The conjugate of $x$ is
$$x^* = x_0 - x_1\ii - x_2\jj - x_3\kk,$$
the modulus of $x$ is
$$|x| = |x^*| = \sqrt{xx^*} = \sqrt{x^*x} = \sqrt{x_0^2 + x_1^2 + x_2^2 + x_3^2},$$
and if $x \not = 0$, then $x^{-1} = {x^* \over |x|^2}$.

\subsection{Quaternion Matrices}

The collections of real, complex and quaternion $m \times n$ matrices are denoted by ${\mathbb R}^{m \times n}$, ${\mathbb C}^{m \times n}$ and ${\mathbb Q}^{m \times n}$, respectively.

A quaternion matrix $A= (a_{ij}) \in {\mathbb Q}^{m \times n}$ can be denoted as
\begin{equation} \label{e1}
A = A_0 + A_1\ii + A_2\jj + A_3\kk,
\end{equation}
where $A_0, A_1, A_2, A_3 \in {\mathbb R}^{m \times n}$.   The transpose of $A$ is $A^\T = (a_{ji})$. The conjugate of $A$ is $\bar A = (a_{ij}^*)$.   The conjugate transpose of $A$ is $A^* = (a_{ji}^*) = \bar A^T$.   For $A, B \in {\mathbb Q}^{m \times n}$, their inner product is defined as
$$\langle A, B \rangle = {\rm Tr}(A^*B),$$
where ${\rm Tr}(A^*B)$ denotes the trace of $A^*B$.
The Frobenius norm of $A$ is
$$\|A\|_F = \sqrt{\langle A, A \rangle} = \sqrt{{\rm Tr}(A^*A)} = \sqrt{\sum_{i=1}^m \sum_{j=1}^n |a_{ij}|^2}.$$

The following theorem for the QSVD of a quaternion matrix was proved by Zhang \cite{Zh97}.

\begin{Thm} {\bf (Zhang 1997)} \label{t2.1}
Any quaternion matrix $A \in {\mathbb Q}^{m \times n}$ has the following QSVD form
\begin{equation} \label{e2}
X = U\left({\Sigma_r \ O \atop O \ \ O}\right)V^*,
\end{equation}
where $U \in {\mathbb Q}^{m \times m}$ and $V \in {\mathbb Q}^{n \times n}$ are unitary,  and $\Sigma_r$ = diag$\{ \sigma_1, \cdots, \sigma_r\}$ is a real \blue{positive} $r \times r$ diagonal matrix, with $\sigma_1 \ge \cdots \ge \sigma_r$ as the singular values of $A$.
\end{Thm}

The $\ell_1$-norm of $A = (a_{ij}) \in {\mathbb Q}^{m \times n}$ is defined by $\|A\|_1 = \sum_{i=1}^m \sum_{j=1}^n |a_{ij}|$.   The $\ell_\infty$-norm of $A$ is defined by $\|A\|_\infty = \max_{i, j} |a_{ij}|$.   The spectral norm of $A$ is defined as $\|A\|_S = \max \{ \sigma_1, \cdots, \sigma_r \}$.  The nuclear norm of $A$ is defined as
$\|A\|_* = \sum_{i=1}^r \sigma_i$.

The $\ell_0$-norm of $A = (a_{ij}) \in {\mathbb Q}^{m \times n}$ is defined by $\|A\|_0 =$ the number of $a_{ij} \not = 0$.   It is not a true norm as it does not satisfy the triangular inequality of norms, and it is discrete.  However, it plays an important role in sparse color image processing \cite{XYXZN15}.   In some papers, it is called the counting function \cite{Le13}.

The rank of $A \in {\mathbb Q}^{m \times n}$, denoted as rank$(A)$, is equal to the number of nonzero singular values of $A$.  It is also discrete.   Yet it plays an important role in color image inpainting \cite{CXZ19}.   In \cite{HL13}, the rank function of real matrices was shown to be lower semi-continuous.    With a similar argument, it is seen that the rank function of quaternion matrices is also lower semi-continuous.

In Section 7, we will show that the $\ell_0$-norm of $A$ is lower semi-continuous with respect to $A$.  We will also analyze its generalized subdifferential there.

For a quaternion matrix $A = A_0 + A_1\ii + A_2\jj + A_3\kk \in {\mathbb Q}^{m \times n}$, its real representation \cite{WLZZ18} is
$$A^R = \left(\begin{array} {cccc} A_0 & -A_1 & -A_2 & -A_3 \\ A_1 & A_0 & -A_3 & A_2 \\ A_2 & A_3 & A_0 & -A_1 \\ A_3 & -A_2 & A_1 & A_0 \end{array}\right).$$

A color image dataset can be represented as a pure quaternion matrix
$$A = A_1\ii + A_2\jj + A_3\kk \in {\mathbb Q}^{m \times n},$$
where  $A_1, A_2$ and $A_3$ are the three real $m \times n$ matrices.

We use $O_{m \times n}$ to denote the zero matrix in ${\mathbb Q}^{m \times n}$.

More knowledge of quaternion matrices can be found in \cite{WLZZ18, Zh97}.

\section{The General Problem and The Prototype Problem}

We consider matrix component vector ${\bf X} \equiv (W, Y, Z)$, where $W \in {\mathbb Q}^{m_1 \times n_1}$, $Y \in {\mathbb Q}^{m_2 \times n_2}$ and $Z \in {\mathbb Q}^{m_3 \times n_3}$.  We may consider
matrix component vectors with more components.   But three components are enough for illustrating the problem.    The general constrained optimization problem of quaternion matrix variables has the following form
\begin{equation} \label{e5}
\min \left\{ f({\bf X}) : h_j({\bf X}) = 0, j = 1, \cdots, p, g_k({\bf X}) \le 0, k = 1, \cdots, q \right\},
\end{equation}
where $f, h_j, g_k : {\mathbb Q}^{m_1 \times n_1} \times {\mathbb Q}^{m_2 \times n_2} \times {\mathbb Q}^{m_3 \times n_3} \to {\mathbb R}$, for $j = 1, \cdots, p$ and $k = 1, \cdots, q$.

Denote ${\mathbb H} := {\mathbb Q}^{m_1 \times n_1} \times {\mathbb Q}^{m_2 \times n_2} \times {\mathbb Q}^{m_3 \times n_3}$.

Suppose that
$$\begin{array}{rl}
W & = W_0 + W_1\ii + W_2\jj + W_3\kk, \\
Y & = Y_0 + Y_1\ii + Y_2\jj + Y_3\kk, \\
Z & = Z_0 + Z_1\ii + Z_2\jj + Z_3\kk,
\end{array}$$
where $W_i \in {\mathbb R}^{m_1 \times n_1}$, $Y_i \in {\mathbb R}^{m_2 \times n_2}$ and $Z_i \in {\mathbb R}^{m_3 \times n_3}$, for $i = 0, 1, 2, 3$.  Then $f, h_j$ and $g_k$ for $j = 1, \cdots, p$ and $k = 1, \cdots, q$ can be regarded as functions of $W_i, Y_i$ and $Z_i$ for $i = 0, 1, 2, 3$, and denote such functions as  $f^R, h^R_j$ and $g^R_k$ for $j = 1, \cdots, p$ and $k = 1, \cdots, q$.     We call them the real representations of $f, h_j$ and $g_k$ for $j = 1, \cdots, p$ and $k = 1, \cdots, q$.

For ${\bf X} = (W, Y, Z)$, denote $R({\bf X}) = (W_0, W_1, W_2, W_3, Y_0, Y_1, Y_2, Y_3, Z_0, Z_1, Z_2, Z_3)$, and call $R({\bf X})$ the real representation of $\bf X$.

We have the following three different assumptions.

{\bf Normal Assumption}
The functions $f^R, h^R_j$ and $g^R_k$ for $j = 1, \cdots, p$ and $k = 1, \cdots, q$ are locally Lipschitz continuous with respect to $W_i, Y_i$ and $Z_i$ for $i = 0, 1, 2, 3$.

In this case, we say that $f, h_j$ and $g_k$ for $j = 1, \cdots, p$ and $k = 1, \cdots, q$ are locally Lipschitz.   All the norms of quaternion matrices satisfy this assumption.

{\bf Middle Assumption}
The functions $f^R, h^R_j$ and $g^R_k$ for $j = 1, \cdots, p$ and $k = 1, \cdots, q$ are continuously differentiable with respect to $W_i, Y_i$ and $Z_i$ for $i = 0, 1, 2, 3$.

In this case, we say that $f, h_j$ and $g_k$ for $j = 1, \cdots, p$ and $k = 1, \cdots, q$ are continuously differentiable.

{\bf Strong Assumption}
The functions $f^R, h^R_j$ and $g^R_k$ for $j = 1, \cdots, p$ and $k = 1, \cdots, q$ are twice continuously differentiable with respect to $W_i, Y_i$ and $Z_i$ for $i = 0, 1, 2, 3$.

In this case, we say that $f, h_j$ and $g_k$ for $j = 1, \cdots, p$ and $k = 1, \cdots, q$ are twice continuously differentiable.   The squares of the Frobenius norms of quaternion matrices satisfy this assumption.


With the middle or strong assumption, the general constrained optimization problem (\ref{e5})  can be solved by standard optimization methods.   This is somewhat tedious.   In this paper, we attempt to find some concise quaternion forms for studying (\ref{e5}).

We also have a prototype constrained optimization problem
\begin{equation} \label{e6}
\min \left\{ \left\|YZ-W\right\|_F^2 : W_\Omega = D_\Omega \right\},
\end{equation}
where $W, D \in {\mathbb Q}^{m \times n}$, $Y \in {\mathbb Q}^{m \times r}$ and $Z \in {\mathbb Q}^{r \times n}$, $D$ is a quaternion dataset matrix, $\Omega$ denotes the set of $m \times n$ quaternion matrices, $W$ is the observed matrix, $r$ is the upper bound of the rank of the low rank decomposition $YZ$ to approximate $W$.

This constrained optimization problem arises from the color image inpainting problem \cite{CQZX20}.

Comparing with (\ref{e5}), for (\ref{e6}), we have $m_1 = m_2 = m$, $n_1 = n_3 = n$, $n_2 = m_3 = r$,
$p = |\Omega|$, the cardinality of $\Omega$, and $q = 0$.   Clearly, (\ref{e6}) satisfies the strong assumption.


\section{R-Product, R-Linear Independence and First Order Derivative}

Suppose that $A$ and $E$ are two real matrices with the same dimension.   For example, assume that
$A = (a_{ij})$ and $E = (e_{ij})$ are $m \times n$ real matrices.   Define the R-product (real product) of $A$ and $E$ as
$$A \cdot E := \sum_{i=1}^m \sum_{j=1}^n a_{ij}e_{ij}.$$

Consider the matrix component vector space ${\mathbb H} \equiv {\mathbb Q}^{m_1 \times n_1} \times {\mathbb Q}^{m_2 \times n_2} \times {\mathbb Q}^{m_3 \times n_3}$.   

Suppose that ${\bf A} \equiv (B, C, D) \in {\mathbb H}$ and ${\bf H} \equiv (E, F, G) \in {\mathbb H}$.
Assume that $B = B_0 + B_1\ii + B_2\jj + B_3\kk, \cdots, G = G_0 + G_1\ii + G_2\jj + G_3\kk$, where
$B_i, C_i, D_i, E_i, F_i$ and $G_i$ are real matrices with corresponding dimensions.
Define the R-product (real product) of $\bf A$ and $\bf H$ as
$${\bf A} \cdot {\bf H} := \sum_{i=0}^3 (B_i \cdot E_i + C_i \cdot F_i + D_i \cdot G_i).$$

For ${\bf A} \equiv (B, C, D) \in {\mathbb H}$ and ${\bf H} \equiv (E, F, G) \in {\mathbb H}$, we may define the inner product of $\bf A$ and $\bf H$ as
$$\langle {\bf A}, {\bf H} \rangle = \langle B, E \rangle + \langle C, F \rangle + \langle D, G \rangle.$$

We have the following proposition.

\begin{Prop} \label{p4.0}
Suppose that ${\bf A} \equiv (B, C, D) \in {\mathbb H}$ and ${\bf H} \equiv (E, F, G) \in {\mathbb H}$.   Then the R-product of $\bf A$ and $\bf H$ is the real part of $\langle {\bf A}, {\bf H} \rangle$.
\end{Prop}
\begin{proof} Let
$$B = B_0 + B_1\ii + B_2\jj + B_3\kk,\ E = E_0 + E_1\ii + E_2\jj + E_3\kk,$$
$$B_i = (b_{pqi}),\ E_i = (e_{pqi}),\ {\rm for}\ i = 0, 1, 2, 3.$$
Then
\begin{eqnarray*}
&& \langle B, E \rangle = {\rm Tr}(B^*E) \\
& = & \sum_{p=1}^{m_1} \sum_{q=1}^{n_1} \left(b_{pq0} - b_{pq1}\ii - b_{pq2}\jj - b_{pq3}\kk\right)\left(e_{pq0} + e_{pq1}\ii + e_{pq2}\jj + e_{pq3}\kk\right).
\end{eqnarray*}
We see that the real part of $\left(b_{pq0} - b_{pq1}\ii - b_{pq2}\jj - b_{pq3}\kk\right)\left(e_{pq0} + e_{pq1}\ii + e_{pq2}\jj + e_{pq3}\kk\right)$,
$${\rm Re}\left(b_{pq0} - b_{pq1}\ii - b_{pq2}\jj - b_{pq3}\kk\right)\left(e_{pq0} + e_{pq1}\ii + e_{pq2}\jj + e_{pq3}\kk\right) = \sum_{i=0}^3 b_{pqi}e_{pqi}.$$
This implies that the real part of $\langle B, E \rangle$ is $\sum_{i=0}^3 B_i \cdot E_i$.   Hence, the real part of $\langle {\bf A}, {\bf H} \rangle$ is
${\bf A} \cdot {\bf H}$.
\end{proof}

This proposition reveals the meaning of the R-product.   In Section 7, it connects our derivative concept with the subgradient concept of Jia, Ng and Song \cite{JNS19}.

Suppose that ${\bf A^{(j)}} \equiv (A^{(j)}, B^{(j)}, C^{(j)}) \in {\mathbb H}$ for $j = 1, \cdots, p$.   We say that $\left\{ {\bf A^{(j)}} : j = 1, \cdots, p \right\}$ is R-linearly dependent if there are real numbers $\alpha_j$ for $j = 1, \cdots, p$ such that some of them are nonzero and
$$\sum_{j=1}^p \alpha_j {\bf A^{(j)}} = {\bf O}.$$

Suppose that $f({\bf X}) : {\mathbb H} \to {\mathbb R}$ satisfies the normal assumption, with ${\bf X} = (W, Y, Z)$, and $f$ is differentiable at $\bf X$ with respect to $W_i, Y_i$ and $Z_i$ for $i = 0, 1, 2, 3$.  Then we define the partial derivatives and gradient of $f$ at ${\bf X}$ as:
$${\partial \over \partial W} f({\bf X}) = {\partial \over \partial W_0} f({\bf X}) + {\partial \over \partial W_1} f({\bf X})\ii + {\partial \over \partial W_2} f({\bf X})\jj + {\partial \over \partial W_3} f({\bf X})\kk,$$
$${\partial \over \partial Y} f({\bf X}) = {\partial \over \partial Y_0} f({\bf X}) + {\partial \over \partial Y_1} f({\bf X})\ii + {\partial \over \partial Y_2} f({\bf X})\jj + {\partial \over \partial Y_3} f({\bf X})\kk,$$
$${\partial \over \partial Z} f({\bf X}) = {\partial \over \partial Z_0} f({\bf X}) + {\partial \over \partial Z_1} f({\bf X})\ii + {\partial \over \partial Z_2} f({\bf X})\jj + {\partial \over \partial Z_3} f({\bf X})\kk,$$
$$\nabla f({\bf X}) = \left({\partial \over \partial W} f({\bf X}), {\partial \over \partial Y} f({\bf X}), {\partial \over \partial Z} f({\bf X})\right).$$

We see that $\nabla f({\bf X}) \in {\mathbb H}$.

Define the directional derivative of $f$ at ${\bf X} = (W, Y, Z) \in {\mathbb H}$ in the direction ${\bf \Delta X} = (\Delta W, \Delta Y, \Delta Z) \in {\mathbb H}$ as
$$f'({\bf X}; {\bf \Delta X}) = \lim_{t \to 0 \atop t \in {\mathbb R}} {f({\bf X} + t{\bf \Delta X}) - f({\bf X}) \over t}.$$
Note that while the gradient of $f$ is in ${\mathbb H}$, the directional derivative of $f$ is real.  They are connected via the R-product operation in ${\mathbb H}$.

\begin{Prop} \label{p4.1}
Suppose that ${\bf X} = (W, Y, Z) \in {\mathbb H}$ and ${\bf \Delta X} = (\Delta W, \Delta Y, \Delta Z) \in {\mathbb H}$.   Then
$$f'({\bf X}; {\bf \Delta X}) = \nabla f({\bf X}) \cdot {\bf \Delta X}.$$
Furthermore, if $f$ satisfies the middle assumption, then we have
$$f({\bf X} + {\bf \Delta X}) = f({\bf X}) + \nabla f({\bf X}) \cdot {\bf \Delta X} + o(\|{\bf \Delta X}\|_F).$$
\end{Prop}
\begin{proof} We may regard $f({\bf X})$ as a function of $W_i, Y_i$ and $Z_i$, $i = 0, 1, 2, 3$.   Then
\begin{eqnarray*}
&& f'({\bf X}; {\bf \Delta X})\\
& = & \sum_{i=0}^3 \left[{\partial \over \partial W_i}f({\bf X})\Delta W_i + {\partial \over \partial Y_i}f({\bf X})\Delta Y_i + {\partial \over \partial Z_i}f({\bf X})\Delta Z_i\right]\\
& = & \nabla f({\bf X}) \cdot {\bf \Delta X}.
\end{eqnarray*}
Furthermore,
\begin{eqnarray*}
&& f({\bf X} + {\bf \Delta X}) - f({\bf X})\\
& = & \sum_{i=0}^3 \left[{\partial \over \partial W_i}f({\bf X})\Delta W_i + o(\|{\Delta W_i }\|) + {\partial \over \partial Y_i}f({\bf X})\Delta Y_i + o(\|{\Delta Y_i }\|) + {\partial \over \partial Z_i}f({\bf X})\Delta Z_i + o(\|{\Delta Z_i }\|) \right] \\
& = & \nabla f({\bf X}) \cdot {\bf \Delta X} + o(\|{\bf \Delta X}\|_F).
\end{eqnarray*}
\end{proof}

Now, we may study the first order optimality conditions of (\ref{e5}).

\begin{Thm} \label{t4.2}
Suppose that the functions $f, h_j$ and $g_k$ for $j = 1, \cdots, p$ and $k = 1, \cdots, q$ satisfy the middle assumption.   Assume that ${\bf X}^{\#} = \left(W^{\#}, Y^{\#}, Z^{\#}\right) \in {\mathbb H}$ is an optimal solution of (\ref{e5}).  If
\begin{equation} \label{e7}
\left\{ \nabla h_j\left({\bf X}^{\#}\right), j = 1, \cdots, p \right\} \bigcup \left\{ \nabla g_k\left({\bf X}^{\#}\right): g_k({\bf X^{\#}}) = 0, 1 \le k \le q \right\}
\end{equation}
is R-linearly independent, then there are Langrangian multipliers $\lambda_j, \mu_k \in {\mathbb R}$ for $j = 1, \cdots, p$ and $k = 1, \cdots, q$, such that
\begin{equation} \label{e8}
\nabla f\left({\bf X}^{\#}\right) + \sum_{j=1}^p \lambda_j \nabla h_j\left({\bf X}^{\#}\right) + \sum_{k=1}^p \mu_k \nabla g_k\left({\bf X}^{\#}\right) = {\bf O},
\end{equation}
\begin{equation} \label{e9}
h_j\left({\bf X}^{\#}\right) = 0, \ \ j = 1, \cdots, p,
\end{equation}
\begin{equation} \label{e10}
g_k\left({\bf X}^{\#}\right) \le 0, \mu_k \ge 0, \mu_kg_k\left({\bf X}^{\#}\right)=0, \ \ k = 1, \cdots, q.
\end{equation}
\end{Thm}
\begin{proof} Again, we may regard $f, h_j, g_k$ for $j = 1, \cdots, p$, $k = 1, \cdots, q$ as functions of $W_i, Y_i$ and $Z_i$, $i = 0, 1, 2, 3$.   Then (\ref{e7}) is equivalent to linear independence constraint qualification of such an optimization problem of $W_i, Y_i$ and $Z_i$, $i = 0, 1, 2, 3$.  Then from the first order optimality condition for real constrained optimization, we have (\ref{e8}-\ref{e10}).
\end{proof}

Note that the Langrangian multipliers are real numbers.
We call ${\bf X}^{\#} \in {\mathbb H}$, which satisfies (\ref{e8}), (\ref{e9}) and (\ref{e10}) with some Langrangian multipliers, a stationary point of (\ref{e5}).
We may reduce the R-linear independence condition in Theorem \ref{t4.2} to other constraint qualifications for nonlinear programs.

We now consider the prototype problem (\ref{e6}).   Let $f({\bf X}) \equiv {1 \over 2}\|YZ-W\|_F^2$.  Then by \cite{CQZX20}, we have
\begin{equation} \label{e11}
{\partial \over \partial W} f({\bf X}) = W-YZ,
\end{equation}
\begin{equation} \label{e12}
{\partial \over \partial Y} f({\bf X}) = (YZ-W)Z^*,
\end{equation}
\begin{equation} \label{e13}
{\partial \over \partial Z} f({\bf X})  = Y^*(YZ-W)
\end{equation}
and the following theorem.

\begin{Thm} \label{t4.3}
Assume that ${\bf X}^{\#} = \left(W^{\#}, Y^{\#}, Z^{\#}\right) \in {\mathbb H}$ is an optimal solution of (\ref{e6}).  Then ${\bf X}^{\#}$ is a stationary point of (\ref{e6}), i.e.,
$$\begin{array}{rl}
W^{\#}_{\Omega_C} & = (A^{\#}B^{\#})_{\Omega_C}, \\
\left(A^{\#}B^{\#}-W^{\#}\right)\left(B^{\#}\right)^{*} & = O_{m \times r}, \\
\left(A^{\#}\right)^{*}\left(A^{\#}B^{\#}-W^{\#}\right) & = O_{r \times n}, \\
W^{\#}_{\Omega} & = D_{\Omega},
\end{array}$$
where $\Omega_C$ is the complement set of $\Omega$.
\end{Thm}

We now study the product rule and the chain rule of first order derivatives.

\begin{Thm} \label{t4.4}
Suppose that
$f({\bf X}),  g({\bf X}): {\mathbb H} \to {\mathbb R}$ satisfies the middle assumption, with ${\bf X} = (W, Y, Z)$.
Then
\begin{equation}
\nabla (f({\bf X})g({\bf X})) = f({\bf X})\nabla g({\bf X}) +  g({\bf X})\nabla f({\bf X}).
\end{equation}
\end{Thm}
\begin{proof} We have
\begin{eqnarray*}
&& {\partial \over \partial W} (f({\bf X})g({\bf X})) \\
&=& {\partial \over \partial W_0} (f({\bf X})g({\bf X})) + {\partial \over \partial W_1} (f({\bf X})g({\bf X}))\ii + {\partial \over \partial W_2} (f({\bf X})g({\bf X}))\jj + {\partial \over \partial W_3} (f({\bf X})g({\bf X}))\kk\\
&=& \left[f({\bf X}){\partial \over \partial W_0}g({\bf X}) + g({\bf X}){\partial \over \partial W_0}f({\bf X})\right] + \left[f({\bf X}){\partial \over \partial W_1}g({\bf X}) + g({\bf X}){\partial \over \partial W_1}f({\bf X})\right]\ii \\
&+& \left[f({\bf X}){\partial \over \partial W_2}g({\bf X}) + g({\bf X}){\partial \over \partial W_2}f({\bf X})\right]\jj + \left[f({\bf X}){\partial \over \partial W_3}g({\bf X}) + g({\bf X}){\partial \over \partial W_3}f({\bf X})\right]\kk \\
&=& f({\bf X}){\partial \over \partial W}g({\bf X}) + g({\bf X}){\partial \over \partial W}f({\bf X}).
\end{eqnarray*}
Similarly, we have
$${\partial \over \partial Y} (f({\bf X})g({\bf X})) = f({\bf X}){\partial \over \partial Y}g({\bf X}) + g({\bf X}){\partial \over \partial Y}f({\bf X})$$
and
$${\partial \over \partial Z} (f({\bf X})g({\bf X})) = f({\bf X}){\partial \over \partial Z}g({\bf X}) + g({\bf X}){\partial \over \partial Z}f({\bf X}).$$
Then
\begin{eqnarray*}
&& \nabla (f({\bf X})g({\bf X}))\\
& = & \left({\partial \over \partial W} (f({\bf X})g({\bf X})), {\partial \over \partial Y} (f({\bf X})g({\bf X})), {\partial \over \partial Z} (f({\bf X})g({\bf X}))\right)\\
& = & \left(f({\bf X}){\partial \over \partial W}g({\bf X}) + g({\bf X}){\partial \over \partial W}f({\bf X}), f({\bf X}){\partial \over \partial Y}g({\bf X}) + g({\bf X}){\partial \over \partial Y}f({\bf X}),f({\bf X}){\partial \over \partial Z}g({\bf X}) + g({\bf X}){\partial \over \partial Z}f({\bf X})\right)\\
& = & f({\bf X})\nabla g({\bf X}) +  g({\bf X})\nabla f({\bf X}).
\end{eqnarray*}
\end{proof}

Similarly, we may prove the following theorem.

\begin{Thm} \label{t4.5}
Suppose that
$f({\bf X}): {\mathbb H} \to {\mathbb R}$ satisfies the middle assumption, with ${\bf X} = (W, Y, Z)$,
and $\phi : {\mathbb R} \to {\mathbb R}$ is continuously differentiable.
Then
\begin{equation}
\nabla \phi(f({\bf X})) = \phi'(f({\bf X}))\nabla f({\bf X}).
\end{equation}
\end{Thm}

In \cite{CQZX20}, the Kurdyka- Lojasiewicz inequality was established for real functions of quaternion matrix variables.

\section{Second Order Derivative}

Suppose that $f({\bf X}) \equiv f(W, Y, Z) : {\mathbb H} \to {\mathbb R}$ is twice continuously differentiable in the sense of the strong assumption.   Then the second order derivative of $f$ exists.
We first consider the second order partial derivatives of $f$.  To make an example, we consider ${\partial^2 \over \partial Y \partial W}f({\bf X})$.  It is more convenient to consider
${\partial^2 \over \partial Y \partial W}f({\bf X})\Delta Y$ and
${\partial^2 \over \partial Y \partial W}f({\bf X})\Delta Y \cdot \Delta W$.

Let ${\partial \over \partial W}f({\bf X})$ be as considered in the last section.  Let $\Delta Y \in {\mathbb Q}^{m_2 \times n_2}$.    Suppose that
$${\partial \over \partial W}f(W, Y+\Delta Y, Z) - {\partial \over \partial W}f(W, Y, Z) = \phi(W, Y+\Delta Y, Z) + \eta(W, Y+\Delta Y, Z),$$
where $\phi$ is R-linear in $\Delta Y$ in the sense that for any $\alpha, \beta \in {\mathbb R}$ and
$\Delta Y^{(1)}, \Delta Y^{(2)} \in {\mathbb Q}^{m_2 \times n_2}$,
$$\phi\left(W, Y+\alpha\Delta Y^{(1)} + \beta\Delta Y^{(2)}, Z\right) =
\alpha \phi\left(W, Y+\alpha\Delta Y^{(1)}, Z\right) + \beta \phi\left(W, Y+\alpha\Delta Y^{(2)}, Z\right),$$
and
$$\eta(W, Y+\Delta Y, Z) = o\left(\left\|\Delta Y\right\|_F\right),$$
i.e.,
$$\lim_{\left\|\Delta Y\right\|_F \to 0} {\eta(W, Y+\Delta Y, Z) \over \left\|\Delta Y\right\|_F} = 0.$$
Then we have
$${\partial^2 \over \partial Y \partial W}f({\bf X})\Delta Y = \phi(W, Y+\Delta Y, Z).$$
Later, we will see that we may use this approach to calculate ${\partial^2 \over \partial Y \partial W}f({\bf X})\Delta Y$ and
${\partial^2 \over \partial Y \partial W}f({\bf X})\Delta Y \cdot \Delta W$, for $f$ defined by our prototype problem (\ref{e6}).

We may express the other second order partial derivatives of $f$ similarly.

\begin{Prop} \label{p5.1}
Under the strong assumption, we have
$${\partial^2 \over \partial Y \partial W}f({\bf X})\Delta Y \cdot \Delta W = {\partial^2 \over \partial W \partial Y}f({\bf X})\Delta W \cdot \Delta Y,$$
$${\partial^2 \over \partial Z \partial W}f({\bf X})\Delta Z \cdot \Delta W = {\partial^2 \over \partial W \partial Z}f({\bf X})\Delta W \cdot \Delta Z,$$
and
$${\partial^2 \over \partial Y \partial Z}f({\bf X})\Delta Y \cdot \Delta Z = {\partial^2 \over \partial Z \partial Y}f({\bf X})\Delta Z \cdot \Delta Y.$$
\end{Prop}
\begin{proof}
Let $\Delta W = \Delta W_0 + \Delta W_1\ii + \Delta W_2\jj + \Delta W_3\kk \in {\mathbb Q}^{m_1 \times n_1}$ and $\Delta Y = \Delta Y_0 + \Delta Y_1\ii + \Delta Y_2\jj + \Delta Y_3\kk \in {\mathbb Q}^{m_2 \times n_2}$.  Then ${\partial^2 \over \partial Y \partial W}f({\bf X})\Delta Y \cdot \Delta W$ and ${\partial^2 \over \partial W \partial Y}f({\bf X})\Delta W \cdot \Delta Y$  are real.  We have
$${\partial^2 \over \partial Y \partial W}f({\bf X})\Delta Y \cdot \Delta W
= \sum_{i, j=0}^3 {\partial^2 \over \partial Y_i \partial W_j}f({\bf X})\Delta Y_i \cdot \Delta W_j$$
$$= \sum_{i, j=0}^3 {\partial^2 \over \partial W_j \partial Y_i}f({\bf X})\Delta W_j \cdot \Delta Y_i
= {\partial^2 \over \partial W \partial Y}f({\bf X})\Delta W \cdot \Delta Y.$$
The other two equalities can be proved similarly.
\end{proof}

Then we may define $\nabla^2 f({\bf X})$ by
$${1 \over 2}\nabla^2 f({\bf X}) {\bf \Delta X} \cdot {\bf \Delta X} = {\partial^2 \over \partial Y \partial W}f({\bf X})\Delta Y \cdot \Delta W + {\partial^2 \over \partial Z \partial W}f({\bf X})\Delta Z \cdot \Delta W + {\partial^2 \over \partial Y \partial Z}f({\bf X})\Delta Y \cdot \Delta Z$$
$$+ {1 \over 2}{\partial^2 \over \partial W^2}f({\bf X})\Delta W \cdot \Delta W + {1 \over 2}{\partial^2 \over \partial Y^2}f({\bf X})\Delta Y \cdot \Delta Y + {1 \over 2}{\partial^2 \over \partial Z^2}f({\bf X})\Delta Z \cdot \Delta Z.$$
Here, ${\bf \Delta X} = (\Delta W, \Delta Y, \Delta Z) \in {\mathbb H}$.

If ${1 \over 2}\nabla^2 f({\bf X}) {\bf \Delta X} \cdot {\bf \Delta X} \ge 0$ for any ${\bf \Delta X} \in {\mathbb H}$, then we say that $\nabla^2 f$ is positive semi-definite at ${\bf X}$. If ${1 \over 2}\nabla^2 f({\bf X}) {\bf \Delta X} \cdot {\bf \Delta X} > 0$ for any ${\bf \Delta X} \in {\mathbb H}$ and ${\bf \Delta X} \not = {\bf O}$, then we say that $\nabla^2 f$ is positive definite at ${\bf X}$.

\begin{Prop} \label{p5.2}
Under the strong assumption, we have
\begin{equation} \label{e14}
f({\bf X} + {\bf \Delta X}) = f({\bf X}) + \nabla f({\bf X}) \cdot {\bf \Delta X} + {1 \over 2} \nabla^2 f({\bf X}){\bf \Delta X} \cdot {\bf \Delta X} + o(\|{\bf \Delta X}\|_F^2).
\end{equation}
\end{Prop}
\begin{proof} Again, we may regard $f({\bf X})$ as a function of $W_i, Y_i$ and $Z_i$, $i = 0, 1, 2, 3$.   Then
\begin{eqnarray*}
&& f({\bf X} + {\bf \Delta X}) - f({\bf X}) \\
& = & \sum_{i=0}^3 \left[{\partial \over \partial W_i}f({\bf X})\Delta W_i + {\partial \over \partial Y_i}f({\bf X})\Delta Y_i + {\partial \over \partial Z_i}f({\bf X})\Delta Z_i\right] + \sum_{i, j=0}^3 {\partial^2 \over \partial Y_i \partial W_j}f({\bf X})\Delta Y_i \cdot \Delta W_j\\
& + & \sum_{i, j=0}^3 {\partial^2 \over \partial Z_i \partial W_j}f({\bf X})\Delta Z_i \cdot \Delta W_j
+ \sum_{i, j=0}^3 {\partial^2 \over \partial Y_i \partial Z_j}f({\bf X})\Delta Y_i \cdot \Delta Z_j +{1 \over 2}\sum_{i, j=0}^3 {\partial^2 \over \partial W_i \partial W_j}f({\bf X})\Delta W_i \cdot \Delta W_j\\
& + & {1 \over 2}\sum_{i, j=0}^3 {\partial^2 \over \partial Y_i \partial Y_j}f({\bf X})\Delta Y_i \cdot \Delta Y_j + {1 \over 2}\sum_{i, j=0}^3 {\partial^2 \over \partial Z_i \partial Z_j}f({\bf X})\Delta Z_i \cdot \Delta Z_j\\
& + & \sum_{i=0}^3 \left[o(\|\Delta W_i\|_F^2) + o(\|\Delta Y_i\|_F^2) + o(\|\Delta Z_i\|_F^2) \right] \\
& = & \nabla f({\bf X}) \cdot {\bf \Delta X} + {1 \over 2} \nabla^2 f({\bf X}){\bf \Delta X} \cdot {\bf \Delta X} + o(\|{\bf \Delta X}\|_F^2).
\end{eqnarray*}
We have (\ref{e14}).
\end{proof}

We now consider the function $f$ in the prototype problem (\ref{e6}).  Then, we have the following result.

\begin{Thm} \label{t5.3}
Suppose that $f({\bf X}) \equiv \|YZ-W\|_F^2$, where $W \in {\mathbb Q}^{m \times n}$, $Y \in {\mathbb Q}^{m \times r}$ and $Z \in {\mathbb Q}^{r \times n}$.
Then,
\begin{equation} \label{e15}
{\partial^2 \over \partial W^2}f({\bf X}) \Delta W = \Delta W,
\end{equation}
\begin{equation} \label{e16}
{\partial^2 \over \partial Y \partial W}f({\bf X}) \Delta Y = - \Delta Y Z,
\end{equation}
\begin{equation} \label{e17}
{\partial^2 \over \partial Z \partial W}f({\bf X}) \Delta Z = - Y \Delta Z,
\end{equation}
\begin{equation} \label{e18}
{\partial^2 \over \partial W \partial Y}f({\bf X}) \Delta W = - \Delta W Z^*,
\end{equation}
\begin{equation} \label{e19}
{\partial^2 \over \partial Y^2}f({\bf X}) \Delta Y = \Delta Y ZZ^*,
\end{equation}
\begin{equation} \label{e20}
{\partial^2 \over \partial Z \partial Y}f({\bf X}) \Delta Z = (YZ-X)(\Delta Z)^* + Y(\Delta Z)Z^*,
\end{equation}
\begin{equation} \label{e21}
{\partial^2 \over \partial W \partial Z}f({\bf X}) \Delta W = - Y^* \Delta W,
\end{equation}
\begin{equation} \label{e22}
{\partial^2 \over \partial Y \partial Z}f({\bf X}) \Delta Y = (\Delta Y)^*(YZ-X) + Y^*(\Delta Y)Z,
\end{equation}
\begin{equation} \label{e23}
{\partial^2 \over \partial Y^2}f({\bf X}) \Delta Y = YY^* \Delta Z.
\end{equation}
\end{Thm}
\begin{proof} By (\ref{e11}), we have
$${\partial \over \partial W} f(W, Y, Z) = W-YZ.$$
Thus,
$${\partial \over \partial W} f(W+\Delta W, Y, Z) = W+\Delta W -YZ,$$
$${\partial \over \partial W} f(W+\Delta W, Y, Z) - {\partial \over \partial W} f(W, Y, Z) = \Delta W.$$
This implies (\ref{e15}).
By above, we also have
$${\partial \over \partial W} f(W, Y+\Delta Y, Z) = W -(Y+\Delta Y)Z,$$
$${\partial \over \partial W} f(W, Y+\Delta Y, Z) - {\partial \over \partial W} f(W, Y, Z) = - \Delta YZ.$$
This implies (\ref{e16}).   The formulas (\ref{e17}-\ref{e23}) can be proved similarly.
\end{proof}

We study the second order optimality conditions of (\ref{e5}) now.

Suppose that ${\bf X}^{\#} \in {\mathbb H}$ is a stationary point of (\ref{e5}), i.e., there exist some Langrangian multipliers $\lambda_j, \mu_k \in {\mathbb R}$ for $j = 1, \cdots, p$ and $k = 1, \cdots, q$, such that (\ref{e8}), (\ref{e9}) and (\ref{e10}) are satisfied.   Let
$$I({\bf X}^{\#}) = \left\{ k = 1, \cdots, p : g_k({\bf X}^{\#})= 0 \right\}.$$
We say that ${\bf \Delta X} \in {\mathbb H}$ is in the critical cone $C({\bf X}^{\#})$ of (\ref{e5}) at ${\bf X}^{\#}$, if
\begin{equation} \label{e24}
\nabla f({\bf X}^{\#}) \cdot {\bf \Delta X} = 0,
\end{equation}
\begin{equation} \label{e24}
\nabla h_j({\bf X}^{\#}) \cdot {\bf \Delta X} = 0,\ {\rm for}\ j = 1, \cdots, p,
\end{equation}
\begin{equation} \label{e25}
\nabla g_k({\bf X}^{\#}) \cdot {\bf \Delta X} = 0,\ {\rm for}\ k \in I({\bf X}^{\#}).
\end{equation}

We have the following theorems.

\begin{Thm} \label{t5.4}
Suppose that the functions $f, h_j$ and $g_k$ for $j = 1, \cdots, p$ and $k = 1, \cdots, q$ satisfy the strong assumption.   Assume that ${\bf X}^{\#} = \left(W^{\#}, Y^{\#}, Z^{\#}\right) \in {\mathbb H}$ is an optimal solution of (\ref{e5}), and (\ref{e7}) is R-linearly independent.  Then for any ${\bf \Delta X} \in {\mathbb H}$,
$${1 \over 2}\nabla^2 f({\bf X}) {\bf \Delta X} \cdot {\bf \Delta X} \ge 0.$$
\end{Thm}

\begin{Thm} \label{t5.5}
Suppose that the functions $f, h_j$ and $g_k$ for $j = 1, \cdots, p$ and $k = 1, \cdots, q$ satisfy the strong assumption.   Assume that ${\bf X}^{\#} = \left(W^{\#}, Y^{\#}, Z^{\#}\right) \in {\mathbb H}$ is a stationary point of (\ref{e5}), and for any ${\bf \Delta X} \in {\mathbb H}$,
$${1 \over 2}\nabla^2 f({\bf X}) {\bf \Delta X} \cdot {\bf \Delta X} > 0.$$
Then ${\bf X}^{\#}$ is an optimal minimizer of (\ref{e5}).
\end{Thm}

By regarding $f, h_j, g_k$ for $j = 1, \cdots, p$, $k = 1, \cdots, q$ as functions of $W_i, Y_i$ and $Z_i$, $i = 0, 1, 2, 3$, from the second order optimality necessary condition and sufficient condition for real constrained optimization, we have Theorems \ref{t5.4} and \ref{t5.5}.






\section{Convex Functions of Quaternion Matrix Variables}

Jia, Ng and Song \cite{JNS19} introduced subgradients of norms of quaternion matrix variables.   Thus, they studied convex functions of quaternion matrix variables.

Suppose that $f({\bf X}) : {\mathbb H} \to {\mathbb R}$.   A natural definition for $f$ to be convex is as follows.   We say that $f$ is a convex function if for any ${\bf X, \bf \hat X} \in {\mathbb H}$, and any $t \in {\mathbb R}$, $0 \le t \le 1$, we have
$$f(t{\bf X}+(1-t){\bf \hat X}) \le tf({\bf X}) + (1-t)f({\bf \hat X}).$$
It is possible to use the modern convex analysis terminology, epigraphs, to define convex functions of quaternion matrix variables.   Here, we use the classical definition to make the definition, such that it is more convenient for engineering readers.

Then, a question is: Is $f$ a convex function if and only if $f$ is convex when $f$ is regarded as a function of $W_i, Y_i$ and $Z_i$, $i = 0, 1, 2, 3$?

\begin{Prop} \label{p6.1}
Suppose that $f({\bf X}) : {\mathbb H} \to {\mathbb R}$.  Then $f$ is a convex function if and only if $f$ is convex when $f$ is regarded as a function of $W_i, Y_i$ and $Z_i$, $i = 0, 1, 2, 3$.
\end{Prop}
\begin{proof}  Let ${\bf X} = \left(W,  Y,  Z\right), {\bf \hat X} = \left(\hat W, \hat Y, \hat Z\right) \in {\mathbb H}$, $t \in {\mathbb R}$ and $0 \le t \le 1$.   Then
$$t{\bf X}+(1-t){\bf \hat X} = \left(tW+(1-t)\hat W, tY+(1-t)\hat Y, tZ+(1-t)\hat Z\right),$$
\begin{eqnarray*}
&& tW+(1-t)\hat W \\
&=& \left(tW_0+(1-t)\hat W_0\right)+\left(tW_1+(1-t)\hat W_1\right)\ii+\left(tW_2+(1-t)\hat W_2\right)\jj+\left(tW_3+(1-t)\hat W_3\right)\kk,
\end{eqnarray*}
\begin{eqnarray*}
&& tY+(1-t)\hat Y \\
&=& \left(tY_0+(1-t)\hat Y_0\right)+\left(tY_1+(1-t)\hat Y_1\right)\ii+\left(tY_2+(1-t)\hat Y_2\right)\jj+\left(tY_3+(1-t)\hat Y_3\right)\kk,
\end{eqnarray*}
\begin{eqnarray*}
&& tZ+(1-t)\hat Z \\
&=& \left(tZ_0+(1-t)\hat Z_0\right)+\left(tZ_1+(1-t)\hat Z_1\right)\ii+\left(tZ_2+(1-t)\hat Z_2\right)\jj+\left(tZ_3+(1-t)\hat Z_3\right)\kk.
\end{eqnarray*}
From here, we may conclude that $f$ is a convex function if and only if $f$ is convex when $f$ is regarded as a function of $W_i, Y_i$ and $Z_i$, $i = 0, 1, 2, 3$.
\end{proof}

We now define subgradients and subdifferentials by R-product.   Suppose that $f({\bf X}) : {\mathbb H} \to {\mathbb R}$, and ${\bf \bar X} = \left(\bar W, \bar Y, \bar Z\right) \in {\mathbb H}$.   Let ${\bf G} = \left(A, B,  C\right) \in {\mathbb H}$.   We say that ${\bf G}$ is a subgradient of $f$ at ${\bf \bar X}$ if for any ${\bf X} = \left(W, Y, Z\right) \in {\mathbb H}$, we have
$$f({\bf X}) \ge f({\bf \bar X}) + {\bf G} \cdot ({\bf X}-{\bf \bar X}).$$
The set of all subgradients of $f$ at  ${\bf \bar X}$ is called the subdifferential of $f$ at  ${\bf \bar X}$ and denoted as $\partial  f({\bf \bar X})$.

By the definition of R-product, we see that ${\bf G}$ is a subgradient of $f$ at ${\bf \bar X}$ if and only if $R({\bf G})$ is a subgradient of $f^R$ at $R({\bf \bar X})$.

This definition is slightly different from the definition of Jia, Ng and Song \cite{JNS19} in face.   They used the real part of the inner product of ${\bf G}$ and ${\bf X}-{\bf \bar X}$, instead of their R-product here.   By Proposition \ref{p4.0}, these two definitions are the same.   This also reveals that the subgradient concept introduced in \cite{JNS19} can be regarded as subgradients of the norms of real matrix variables.

From Proposition \ref{p6.1}, the definition of R-product and the knowledge of convex functions of real variables, we have the following proposition.

\begin{Prop} \label{p6.2}
Suppose that $f({\bf X}) : {\mathbb H} \to {\mathbb R}$ is a convex function.   Then for any ${\bf \bar X} = \left(\bar W, \bar Y, \bar Z\right) \in {\mathbb H}$, the subdifferential $\partial  f({\bf \bar X})$ is a nonempty, convex and compact set in ${\mathbb H}$.  The subdifferential $\partial  f({\bf \bar X})$ is a singleton if and only if $f$ is differentiable at $\bf \bar X$.   In this case, $\nabla f({\bf \bar X})$ is the unique subgradient of $f$ at $\bf \bar X$.
\end{Prop}


By Proposition \ref{p6.1} and the knowledge of convex functions of real variables, we have the following proposition.

\begin{Prop} \label{p6.3}
Suppose that $f({\bf X}) : {\mathbb H} \to {\mathbb R}$ satisfies the strong assumption.   Then $f$ is convex if and only if $\nabla^2 f$ is positive semi-definite at any ${\bf X} \in {\mathbb H}$.
\end{Prop}

\section{A Sparse Color Image Denoising Model}

We now consider a sparse color image denoising (SCID) model
\begin{equation} \label{e27}
\min \left\{ f({\bf X}) \equiv {1 \over 2}\|L(Y + Z) - D\|_F^2 + \lambda \|Z\|_0 : {\rm rank}(Y) \le r \right\},
\end{equation}
where $D \in {\mathbb Q}^{m \times n}$ is an observation quaternion matrix of the color image, $Y \in {\mathbb Q}^{m \times n}$ is a low rank quaternion matrix to approximate $D$, $r$ is a prescribed integer for the upper bound of the rank of $Y$, $Z \in {\mathbb Q}^{m \times n}$ is the color image noise to be detected, $\lambda > 0$ is a prescribed parameter, $L : {\mathbb Q}^{m \times n} \to {\mathbb Q}^{m \times n}$ is a linear operator, for example, a projection operator to indicate the observed area, ${\bf X} = (Y, Z)$ is the quaternion matrix vector variable, $f : {\mathbb M}^{m \times n} \equiv {\mathbb Q}^{m \times n} \times {\mathbb Q}^{m \times n} \to {\mathbb R}$ is the objective function.   Assume that
$$Y = Y_0 + Y_1\ii + Y_2\jj + Y_3\kk,\ Z = Z_0 + Z_1\ii + Z_2\jj + Z_3\kk,\ D = D_0 + D_1\ii + D_2\jj + D_3\kk.$$
The task of this section is to analyze the optimality conditions of (\ref{e27}) to pave the way for further study on similar color image models.  As the $\ell_0$-norm and the rank function are not continuous, we have to develop general subdifferential calculus for proper functions of quaternion matrix variables first.

\subsection{Generalized Subdifferentials of Lower Semi-Continuous Functions of Quaternion Matrix Variables}

The generalized subdifferential calculus for real functions of real variables can be found in standard references \cite{Cl83, Mo06, RW09}.    In this subsection, we extend it to proper functions of quaternion matrix variables.

Denote ${\bar {\mathbb R}} \equiv {\mathbb R} \cup \{ +\infty \}$.    For a function $h({\bf X}) : {\mathbb M} \to {\bar {\mathbb R}}$, denote its domain as dom$(h) = \left\{ {\bf X} \in {\mathbb M} : h({\bf X}) < +\infty \right\}$.   We say that $h$ is a proper function if its domain is not empty.

Suppose that $h({\bf X}) : {\mathbb M} \to {\bar {\mathbb R}}$ is a proper function, ${\bf \bar X} = \left(\bar Y, \bar Z\right) \in$ dom$(h)$.   Let ${\bf G} = \left(A, B\right) \in {\mathbb M}$.   We say that ${\bf G}$ is a F(r\'echet)-subgradient of $h$ at ${\bf \bar X}$ if
$$\liminf_{{\bf X} \to {\bf \bar X}, {\bf X} \not = {\bf \bar X}} {h({\bf X}) - h({\bf \bar X}) - {\bf G} \cdot ({\bf X}-{\bf \bar X}) \over \|{\bf X}- {\bf \bar X}\|} \ge 0.$$
The set of all F-subgradients of $h$ at  ${\bf \bar X}$ is called the Fr\'echet subdifferential of $h$ at  ${\bf \bar X}$ and denoted as $\partial^F  h({\bf \bar X})$.   On the other hand, ${\bf G} \in {\mathbb M}$ is a limiting subgradient of $h$ at  ${\bf \bar X}$ if
$${\bf G} = \lim_{k \to \infty} {\bf G}^k,\ {\bf G}^k \in \partial^F h({\bf X^k}),\ \lim_{k \to \infty} {\bf X}^k = {\bf X}.$$
The set of all limiting subgradients of $h$ at  ${\bf \bar X}$ is called the limiting subdifferential of $h$ at  ${\bf \bar X}$ and denoted as $\partial^L  h({\bf \bar X})$.

By the definition of R-product, we see that ${\bf G}$ is a F-subgradient (limiting subgradient) of $h$ at ${\bf \bar X}$ if and only if $R({\bf G})$ is a F-subgradient (limiting subgradient) of $h^R$ at $R({\bf \bar X})$.  Here, $h$  is regarded as a function of $Y_i$ and $Z_i$ for $i = 0, 1, 2, 3$, and denote such a function as  $h^R$, and
for ${\bf X} = (Y, Z)$, denote  $R(Y) = (Y_0, Y_1, Y_2, Y_3)$, $R(Z)= (Z_0, Z_1, Z_2, Z_3)$ and $R({\bf X}) = (R(Y), R(Z))$.   By (3) of \cite{Le13}, we have
$$\partial^F  h({\bf \bar X}) \subset \partial^L  h({\bf \bar X}).$$

By Theorem 1 of \cite{Le13}, we have the following theorem.

\begin{Thm} \label{t7.1}
Let $A = (a_{ij}) \in {\mathbb Q}^{m \times n}$.   Then
$$\partial^F \| A \|_0 = \partial^L \| A \|_0 = {\mathbb Q}^{m\times n}_{\Gamma^c_A},$$
where ${\mathbb Q}^{m\times n}_{\Gamma^c_A} = \left\{ B \in {\mathbb Q}^{m \times n} : B_{\Gamma_A} = O \right\}$, $\Gamma_A = \left\{ (i, j) : \ a_{ij} \not = 0 \right\}$ is the support of $A$, and $\Gamma^c_A$ is the complementary set of $\Gamma_A$.
\end{Thm}

By using Theorem 4 of \cite{Le13}, we may also characterize the generalized subdifferential of the rank function of $A$.

\subsection{The Feasibility Set of The SCID Optimization Problem}

In this subsection, we study the feasibility set of the SCID optimization problem (\ref{e27}):
$$S = \left\{ Y \in {\mathbb Q}^{m \times n} : {\rm rank}(Y) \le r \right\}.$$
Let
$$R(S) = \left\{ R(Y) : Y \in S  \right\}.$$
In this paper, we say that $R(S)$ is the real representation set of $S$.  Note that it is different from
$\left\{ Y^R : Y \in S  \right\}$.

Recall that a set in a real finite-dimensional space is called a semi-algebraic set if it is defined by a set of polynomial equations, and a real function in that space is called semi-algebraic if its graph is semi-algebraic \cite{ABS13, HL13a}.

\begin{Prop} \label{p7.2}
The set $S$ is closed, and its real representation set $R(S)$ is closed and semi-algebraic.
\end{Prop}
\begin{proof}
Suppose that $Y^k \in S$ and $Y^k \to Y$.   Denote $r_k:=  {\text{rank}}(Y^k)$. Thus $r_k\leq r$. By Theorem \ref{t2.1}, there exist unitary matrices $U^k \in {\mathbb Q}^{m \times m}$ and $V^k \in {\mathbb Q}^{n \times n}$, and real positive diagonal matrices $\Sigma_{r_k}^k = {\text{diag}}\left(\sigma_{k,1}, \ldots, \sigma_{k,r_k}\right)$  such that
$$Y^k = U^k\left({\Sigma_{r_k}^k \ O \atop O \ \ O}\right)(V^k)^*.$$
For each $k$, introduce a diagonal matrix $\Sigma_r^k ={\text{diag}}\left(\sigma_{k,1}, \ldots, \sigma_{k,r_k}, 0,\ldots, 0\right)\in {\mathbb R}^{r\times r}$. We can rewrite $Y^k$  as
$$Y^k = U^k\left({\Sigma_{r}^k \ O \atop O \ \ O}\right)(V^k)^*.$$ 
This is also equivalent to
\begin{equation}\label{rewrite}
\left(U^k\right)^*Y^k V^k = \left({\Sigma_r^k \ O \atop O \ \ O}\right).
\end{equation}

Since $U^k$ and $V^k$ are unitary, $\{ U^k \}$ and $\{ V^k \}$ are bounded. Thus, $\{ U^k \}$ and $\{ V^k \}$ have limiting points. Without loss of generality, we may assume that
$U^k \to U$ and $V^k \to V$. Then $U$ and $V$ are unitary. Taking the limit on both sides of \eqref{rewrite} as $k\rightarrow +\infty$, we can find some nonnegative diagonal matrix $\Sigma_r \in {\mathbb R}^{r\times r}$ such that
$$
U^*YV = \left({\Sigma_r \ O \atop O \ \ O}\right).
$$
That is, $Y = U\left({\Sigma_r \ O \atop O \ \ O}\right) V^*$, which indicates that rank$(Y) \le r$. Thus, $Y\in S$, and the desired closedness of $S$ is obtained.

Since $S$ is closed, the set $R(S)$ is also closed.   By Theorem 1.8.4 of \cite{WLZZ18}, if the rank of $A \in {\mathbb Q}^{m \times n}$ is $r$, then the rank of its real representation $A^R$ is $4r$. Note that
\begin{eqnarray*}
R(S) & = & \left\{R(Y) : Y \in  {\mathbb Q}^{m \times n} : {\rm rank}(Y) \le r  \right\}\\
& = & \left\{R(Y) : Y \in {\mathbb Q}^{m \times n} : {\rm rank}(Y^R) \le 4r  \right\}.
\end{eqnarray*}
Since $R(S)$ is characterized by vanishing of all $(4r+1) \times (4r+1)$ minors of $Y^R$, we conclude that the set $R(S)$ is semi-algebraic.
\end{proof}

For a nonempty closed set $\Omega$, the indicator function with respect to $\Omega$, denoted as $\delta_\Omega$ is defined by
$$\delta_\Omega(\vx) = \left\{\begin{array}{ll} 0, & {\rm if}\ \vx \in \Omega; \\ +\infty, & {\rm otherwise}. \end{array}
\right. $$

Let $A \in {\mathbb Q}^{m \times n}$ and $B \to A$.   Then there is $\delta > 0$ such that for all $B \in {\mathbb Q}^{m \times n}$ satisfying $\| B-A \|_F \le \delta$, we have $\|B \|_0 \ge \| A \|_{\blue{0}}$.   This shows that the $\ell_0$-norm function of quaternion matrices is lower semi-continuous.

\begin{Prop} \label{p7.3}
Suppose that $p : {\mathbb M} \to {\mathbb R} \cup \{ +\infty \}$ is defined by $p({\bf X}) = \lambda \|Z\|_0 + \delta_S(Y)$, where ${\bf X} = (Y, Z)$.   Let $p^R(R(Y), R(Z)) \equiv p({\bf X})$. Then,

(i) $p$ is a proper lower semi-continuous function, and $p^R$ is a semi-algebraic function;

(ii) the limiting subdifferential of $h$ takes the form of
$$\partial^L p(Y, Z) = N_S({Y}) \times {\mathbb Q}^{m\times n}_{\Gamma^c_Z},$$
where $N_S(Y)$ is the normal cone with respect to $S$ at $Y$.
\end{Prop}
\begin{proof}  (i) For any $Y \in S$ and $Z \in {\mathbb Q}^{m \times n}$, $p({\bf X})$ is nonnegative and finite valued.  Thus, $p$ is proper.    Since the $\ell_0$-norm function of quaternion matrices is lower semi-continuous, and $S$ is closed by Proposition \ref{p7.2}, $p$ is lower semi-continuous.   By Proposition \ref{p7.2}, $R(S)$ is semi-algebraic.  The $\ell_0$-norm function has a piecewise linear graph, see \cite{ABS13}.  Thus, $p^R$ is a semi-algebraic function.

(ii) Note that $p$ is separable in $Y$ and $Z$.   Thus,
\begin{eqnarray*}
\partial^L p(Y, Z) & =& \partial^L \delta_S(Y) \times \partial^L(\lambda\|Z\|_0)\\
 &=& N_S({Y}) \times {\mathbb Q}^{m\times n}_{\Gamma^c_Z},
 \end{eqnarray*}
 where the second equality comes from Theorem \ref{t7.1}.
\end{proof}

The semi-algebraic property of $p^R$ will not be used in this paper, yet it is very important in convergence analysis of first-order algorithms for solving (\ref{e27}) \cite{ABS13}.

\subsection{Stationarity}

The SCID model (\ref{e27}) is a rank-constrained sparse optimization problem of quaternion matrix variables.  Inspired by the sparse optimization \cite{BE13} and rank-constrained optimization \cite{LXS19}, we may introduce stationary point concepts for (\ref{e27}).

We first introduce proximal mapping for a lower semi-continuous function $h : {\mathbb M} \to {\mathbb R} \cup \{ +\infty \}$.   We denote it as Prox$_h$.  It is defined as
$${\rm Prox}_h({\bf X}) : = {{\rm arg}\min}_{{\bf W} \in {\mathbb M}}  \left\{ h({\bf W}) + {1 \over 2}\|{\bf W} - {\bf X}\|_F^2 \right\}.$$

Then we define $\beta-$stationary points and stationary points for (\ref{e27}).   Let
\begin{equation} \label{e28}
h({\bf X}) = {1 \over 2}\|L(Y+Z) - D\|_F^2.
\end{equation}
Let $\beta$ be a positive number, ${\bf \bar X} = (Y, Z) \in {\mathbb M}$.   We say that ${\bf \bar X}$ a $\beta-$stationary point of (\ref{e27}) if
\begin{equation}\label{def-stationary}\bar Y \in \Pi_S({\bar Y}- \beta\nabla_Y h({\bf \bar X})),\ \bar Z \in {\rm Prox}_{\beta\lambda\|\cdot\|_0}(\blue{\bar{Z}} - \beta\nabla_Zh({\bf \bar X})\blue{)};\end{equation}
we  say that ${\bf \bar X}$ a stationary point of (\ref{e27}) if
$$\nabla_Y h({\bf \bar X}) \in N_S({\bar Y}),\ \nabla_Z h({\bf X}) \in {\mathbb Q}^{m\times n}_{\Gamma^c_A}.$$
Here, $\Pi_S$ is the projection operator with respect to $S$.

By (\ref{e13}), we have
\begin{equation} \label{e29}
\nabla h({\bf X}) =\left[{\partial \over \partial Y}h({\bf X}), {\partial \over \partial Z}h({\bf X})\right] = \left[L^*\left(L(Y+Z)-D\right), L^*\left(L(Y+Z)-D\right)\right],
\end{equation} where $L^*$ stands for the adjoint of $L$ in the sense that
$$L(Y)\cdot W = Y \cdot L^*(W), ~~\forall~ Y, W\in {\mathbb Q}^{m\times n}.$$

\begin{Prop} \label{p7.4}
There is a gradient Lipschitz constant $L_h = \sqrt{2\|L^*L\|} > 0$ (here $\|L^*L\|$ is the spectral norm of the linear operator $L^*L$) such that for any ${\bf X} = (Y, Z), {\bf \bar X} = (\bar Y, \bar Z) \in {\mathbb M}$,
$$\| \nabla h({\bf X}) - \nabla h({\bf \bar X}) \|_F \le L_h \|{\bf X} -{\bf \bar X}\|_F,$$
i.e.,
$$\| \nabla h^R(R({\bf X})) - \nabla h^R(R({\bf \bar X})) \|_F \le L_h \|R({\bf X}) -R({\bf \bar X})\|_F.$$
\end{Prop}
\begin{proof} By (\ref{e29}) and the equivalence relation between $(h, {\bf X})$ and $(h^R, R({\bf X}))$, we have the conclusions.
\end{proof}

\begin{Thm} \label{t7.5}
For the SCID model (\ref{e27}), we have the following conclusions:

(i) Any local minimizer is a stationary point;

(ii) Any $\beta-$stationary point for $\beta > 0$ is a stationary point;

(iii) Any global minimizer ${\bf X^*} = (Y^*, Z^*)$ is a $\beta-$stationary point for $\beta \in (0, {1 \over L_h})$; furthermore, $\Pi_S(Y^*- \beta\nabla_Y f({\bf X^*}))\times$ Prox$_{\beta\lambda\|\cdot\|_0}(Z^*- \beta\nabla_Zf({\bf X^*}))$ is a singletons;

(iv) If ${\bf X^*} = (Y^*, Z^*)$ is a $\beta-$stationary point and rank$(Y^*)<r$, then ${\bf X^*}$ is a local minimizer.

\end{Thm}
\begin{proof}
(i) Let ${\bf \bar X} = (\bar Y, \bar Z)$ be a local minimizer of (\ref{e27}).  Rewrite problem (\ref{e27}) as
\begin{equation} \label{e30}
\min_{{\bf X}} h({\bf X}) + p({\bf X}),
\end{equation}
where $h$ is defined by (\ref{e28}), $p$ is defined in Proposition \ref{p7.3}.  Note that the quaternion matrix optimization problem (\ref{e30}) is equivalent to the real matrix optimization problem
\begin{equation} \label{e31}
\min_{R({\bf X})} h^R(R({\bf X})) + p^R(R({\bf X})).
\end{equation}
Apply the generalized Fermat rule \cite[Theorem 10.1]{RW09} to the real matrix optimization problem (\ref{e31}).  Since ${\bf G}$ is a F-subgradient (limiting subgradient) of $h$ or $p$ at ${\bf \bar X}$ if and only if $R({\bf G})$ is a F-subgradient (limiting subgradient) of $h^R$ or $p^R$ at $R({\bf \bar X})$,
we have the desired result.

(ii) Let ${\bf \bar X} = (\bar Y, \bar Z)$ be a $\beta-$stationary point of (\ref{e27}).   Note that
$$\bar Y \in \Pi_S({\bar Y}- \beta\nabla_Y h({\bf \bar X}))$$
is equivalent to
$$R(\bar Y) \in \Pi_{R(S)}(R({\bar Y})- \beta\nabla_{R(Y)} h^R(R({\bf \bar X})),$$
which is further equivalent to
$$R(\bar Y) = {\rm argmin}_{R(Y)} \left\{{1 \over 2}\left\|R(Y) - (R({\bar Y})- \beta\nabla_{R(Y)} h^R(R({\bf \bar X})))\right\|_F^2 + \delta_{R(S)}(Y)\right\}.$$
Combining this with the generalized Fermat rule \cite[Theorem 10.1]{RW09} and the fact that $$\partial_{R(y)} \delta_{R(S)}(R(Y)) = N_{R(S)}(R({\bar Y}))$$
\cite[Exercise 8.14]{RW09}, we conclude that
$$O \in R(\bar Y)- (R({\bar Y})- \beta\nabla_{R(Y)} h^R(R({\bf \bar X})))+ N_{R(S)}({\bar Y}),$$
which is exactly
\begin{equation} \label{e32}
\nabla_{R(Y)} h^R(R({\bf \bar X})) \in N_{R(S)}(R({\bar Y})),
\end{equation}
due to the conic property of $N_{R(S)}(R({\bar Y}))$.   Note that (\ref{e32}) is equivalent to
\begin{equation} \label{e33}
\nabla_Y h({\bf \bar X}) \in N_S({\bar Y}).
\end{equation}

On the other hand, the relation
$$\bar Z \in {\rm Prox}_{\beta\lambda\|\cdot\|_0}(\bar Z - \beta\nabla_Zh({\bf \bar X})\blue{)}$$
indicates that
$$\nabla_Z h({\bf \bar X}) \in {\mathbb Q}^{m\times n}_{\Gamma^c_A}.$$
Combining with (\ref{e33}), we have Conclusion (ii).

(iii) Denote $R_1^* = \nabla_{R(Y)} h^R(R(Y^*), R(Z^*))$ and $R_2^* = \nabla_{R(Z)} h^R(R(Y^*), R(Z^*))$.   By the Lipschitz continuity of $\nabla h^R$ and the descent lemma \cite{Be99}, we have
\begin{eqnarray*}
&& h^R(R(Y), R(Z)) \\ & \le & h^R(R(Y'), R(Z')) - \langle \nabla h^R(R(Y'), R(Z')), (R(Y) - R(Y'), R(Z) - R(Z')  \rangle \\ &+& {L_h \over 2}\left\|(R(Y) - R(Y'), R(Z) - R(Z'))\right\|_F^2,
\end{eqnarray*}
for any $Y, Y', Z, Z' \in {\mathbb Q}^{m \times n}$.   This is equivalent to
\begin{equation} \label{e35}
h(Y, Z) \le h(Y', Z') + \nabla h(Y', Z') \cdot (Y-Y', Z-Z') + {L_h \over 2}\|(Y-Y', Z-Z')\|_F^2,
\end{equation}
for any $Y, Y', Z, Z' \in {\mathbb Q}^{m \times n}$.   We then prove the both parts simultaneously.  Assume on the contrary that $(Y^*, Z^*)$ is not a $\beta-$stationary point of problem (\ref{e27}) for some $\beta < {1 \over L_h}$.  Then there is $(Y_0, Z_0) \not = (Y^*, Z^*)$ such that
$$(Y_0, Z_0) \in \Pi_S(Y^* - \beta R_1^*)\times {\rm Prox}_{\beta \lambda \|\cdot \|_0} (Z^* - \beta R_2^*).$$
The inclusion $Y_0 \in \Pi_S(Y^*-\beta R_1^*)$ indicates that
$$\| Y_0 - (Y^*-\beta R_1^*)\|_F^2 \le \|Y^* - Y^*-\beta R_1^*)\|_F^2.$$
Thus,
\begin{equation} \label{e36}
R_1^* \cdot (Y_0 - Y^*) \le -{1 \over 2\beta}\|Y_0 - Y^*\|_F^2.
\end{equation}
On the other hand, the relation $Z_0 \in {\rm Prox}_{\beta \lambda \|\cdot \|_0} (Z^* - \beta R_2^*)$ implies that
$${1 \over 2}\| Z_0 - (Z^* - \beta R_2^*) \|_F^2 + \lambda \beta \|Z_0\|_0 \le {1 \over 2}\| Z^* - (Z^* - \beta R_2^*) \|_F^2 + \lambda \beta \|Z^*\|_0.$$
After simplification, we have
\begin{equation} \label{e37}
\lambda \|Z_0\|_0 + {1 \over 2\beta}\|Y_0-Y^*\|_F^2 + R_2^* \cdot (Y_0-Y^*) \le \lambda \|Z^*\|_0.
\end{equation}
By (\ref{e35})-(\ref{e37}), we have
\begin{equation} \label{e38}
h(Y_0, Z_0) + \lambda \|Z_0\|_0 \le h(Y^*, Z^*) + \lambda \|Z^*\|_0 + \left({L_h \over 2} - {1 \over 2\beta}\right)\|(Y_0-Y^*, Z_0-Z^*)\|_F^2.
\end{equation}
Note that $\beta < {1 \over L_h}$ and $\delta_S(Y_0) = \delta_S(Y^*)$.  Then (\ref{e38}) contradicts the global optimality of $(Y^*, Z^*)$.  Thus, $(Y^*, Z^*)$ is the unique element in $\Pi_S(Y^*- \beta\nabla_Y f({\bf X^*}))\times$ Prox$_{\beta\lambda\|\cdot\|_0}(Z^*- \beta\nabla_Zf({\bf X^*}))$.   This proves Conclusion (iii).

(iv) Since the Eckart-Young-Mirsky low-rank approximation theorem can be applied to the case of quaternion matrices \cite{JNS19}, we can easily verify the implication below:
\begin{equation}\label{low-rank}
Y\in \Pi_{S}(Y+T) ~~\&~~ {\rm rank}(Y)<r ~~\Longrightarrow ~~ T = O.
\end{equation}
If ${\bf X^*} = (Y^*, Z^*)$ is a $\beta-$stationary point with rank$(Y^*)<r$, then by the definition of $\beta$-stationarity in \eqref{def-stationary}, together with \eqref{e29}, we have $\nabla_Y h({\bf X^*})= \nabla_Z h({\bf X^*}) =O$ by employing \eqref{low-rank}. Additionally, the convexity of $h$ yields that for any ${\bf X}\in {\mathbb Q}^{m\times n}$,
\begin{equation}\label{grad0} h({\bf X}) \geq h({\bf X^*}) +  \nabla h({\bf X^*}) \cdot ({\bf X}-{\bf X^*}) = h({\bf X^*}).\end{equation}
Note that $\|Z\|_0$ only takes integer values $0$, $1$, $\ldots$, $mn$. Thus we can find some positive scalar $\epsilon$ such that
$$\|Z\|_0\geq \|Z^*\|_0,~~\forall Z\in N(Z^*, \epsilon),$$
where $N(Z^*, \epsilon):=\left\{Z: \|Z-Z^*\|_F \leq \epsilon\right\}$. Combining with \eqref{grad0}, we can conclude that for any feasible solution ${\bf X}\in N(X^*, \epsilon)$, $h({\bf X})+p({\bf X})\geq h({\bf X^*})+p({\bf X^*})$, that is, ${\bf X^*}$ is a local minimizer. This completes the proof.
\end{proof}

\section{Final Remarks}

In this paper, we introduce first and second order derivatives of real functions of quaternion matrix variables, and established their calculation rules.   Our approach is consistent with the subgradient concept for norms of quaternion matrix variables introduced in \cite{JNS19}.   We established first and second order optimality conditions for constrained optimization problems of real functions in quaternion matrix variables.    Optimization methods can be developed based upon these.

One key tool of our approach is the R-product.   It turns out that the R-product of two quaternion matrices is equal to the real part of the inner product of these two quaternion matrices.   This is not by chance.   As we may form a third order real tensor to a quaternion matrix, there is also an inverse reaction to make the operation results of quaternion matrices to the real field, such as singular values and R-products.

Finally, we introduce the generalized subdifferentials of proper functions of quaternion matrices, and use them to analyze the optimality conditions of a sparse low rank color image denoising model.    This combines the knowledge of quaternion matrices, color image processing and variational analysis.

 We hope that our work is useful to people working with optimization models involving quaternion matrices.

\bigskip

{\bf Acknowledgment}  We are thankful to Prof. Michael Ng for the discussion on subgradients of norms of quaternion matrices, and to Prof. Defeng Sun for the discussion on second order derivatives of real functions with quaternion matrix variables.



\end{document}